\documentclass[10pt,reqno]{amsart}
\usepackage{amsmath,amsfonts,amssymb,amsthm}
\usepackage[alphabetic]{amsrefs}
\usepackage{url}
\usepackage{graphicx,color}
\usepackage{scrextend}
\usepackage{hyperref}
\hypersetup{
    colorlinks=true,
    linkcolor=blue,
    filecolor=magenta,
    urlcolor=blue,
    citecolor=blue,
}
\usepackage{subcaption}

\usepackage{xcolor,colortbl}
\definecolor{grayDisabled}{rgb}{0.75, 0.75, 0.75}
\newcolumntype{d}{>{\columncolor{grayDisabled}}r}

\newtheorem{theorem}{Theorem}[section]
\newtheorem{lemma}[theorem]{Lemma}

\newtheorem*{theorem*}{Theorem}

\theoremstyle{definition}
\newtheorem{definition}{Definition}[section]
\newtheorem{conjecture}[theorem]{Conjecture}

\theoremstyle{remark}

\newcommand{\ra}{\rangle}
\newcommand{\la}{\langle }

\usepackage{caption}

\begin{document}

\title[Lucas sequences in t-uniform simplicial complexes]{Lucas sequences in t-uniform simplicial complexes}

\author[Ioana-Claudia Laz\u{a}r]{
Ioana-Claudia Laz\u{a}r\\
Politehnica University of Timi\c{s}oara, Dept. of Mathematics,\\
Victoriei Square $2$, $300006$-Timi\c{s}oara, Romania\\
E-mail address: ioana.lazar@upt.ro}

\date{}

\hyphenation{i-so-pe-ri-me-tric}

\begin{abstract}

We introduce $t$-uniform simplicial complexes and we show that the lengths of spheres in such complexes are the terms of certain Lucas sequences.
We find optimal constants for the linear isoperimetric inequality in the hyperbolic case.

\hspace{0 mm} \textbf{2010 Mathematics Subject Classification}: 05E45, 20F67, 11B37.

\hspace{0 mm} \textbf{Keywords}: simplicial complex, Lucas sequence, recurrence relation, isoperimetric inequality, minimal filling diagram
\end{abstract}

\pagestyle{myheadings}

\markboth{}{}

\vspace{-10pt}

\maketitle

\section{Introduction}

Isoperimetric inequalities relate the length of closed curves to the infimal area of the discs which they bound.
Every closed loop of length $L$ in the Euclidean plane bounds a disc whose area is less than $L^{2} / 4 \pi$, and this bound is optimal.
Thus one has a quadratic isoperimetric inequality for loops in Euclidean space.
In contrast, loops in real hyperbolic space satisfy a linear isoperimetric inequality: there is a constant $C$ such that every closed loop of length $L$ in hyperbolic space bounds a disc whose area is less than or equal to $C \cdot L$.

With a suitable notion of area, a geodesic space $X$ is $\delta$-hyperbolic if and only if loops in $X$ satisfy a linear isoperimetric inequality (see \cite{BH}, chapter $III.H$, page $417$ and page $419$).
For loops in arbitrary CAT(0) spaces, however, there is a quadratic isoperimetric inequality (see \cite{BH}, chapter $III.H$, page $414$).
Osajda introduced in \cite{O-8loc} a local combinatorial condition called $8$-location implying Gromov hyperbolicity of the universal cover (see \cite{L-8loc}).
A related curvature condition, called $5/9$-condition, also implies Gromov hyperbolicity (see \cite{L-8loc2}).
Both $8$-located complexes and $5/9$-complexes satisfy therefore, under the additional hypothesis of simply connectedness, a linear isoperimetric inequality.

One can also express curvature using a condition called local $k$-largeness which was introduced independently by Chepoi \cite{Ch} (under the name of bridged complexes) and by Januszkiewicz-Swiatkowski \cite{JS1}.
A flag simplicial complex is locally $k$-large if its links do not contain essential loops of length less than $k, k \ge 4$.
Cycles in systolic complexes satisfy a quadratic isoperimetric inequality (see \cite{JS1}).
In \cite{E1} explicit constants are provided presenting the optimal estimate on the area of a systolic disc.
In systolic complexes the isoperimetric function for $2$-spherical cycles (the so called second isoperimetric function) is linear (see \cite{JS2}).
In \cite{ChaCHO} it is shown that meshed graphs (thus, in particular, weakly modular graphs) satisfy a quadratic isoperimetric inequality.

The purpose of the current paper is to show that the lengths of spheres in $t$-uniform simplicial complexes are the terms of certain Lucas sequences.
Using this result we find a connection between the area and the length of spheres in such complexes.
For $t \ge 7$, we find optimal constants for the linear isoperimetric inequality in terms of $t$.
We also study $t$-uniform simplicial complexes for $t \le 6$.

We consider certain loops called spheres which are the "roundest" loops.
Intuitively, any loop of the same length which is not a sphere contains less area.
Therefore we consider only the lengths of spheres, not of all loops.
By this choice we can find the best constant for the isoperimetric inequality.
So we do not prove the isoperimetric inequality for the simplicial complex.
Instead we find the best candidate for the constant of the isoperimetric inequality.
The idea is to get a better understanding of how the isoperimetric inequality behaves as the vertices have more or less neighbours.

\textbf{Acknowledgements}.
The author would like to thank Damian Osajda for introducing her to the subject.
This work was partially supported by the grant $346300$ for IMPAN from the Simons Foundation and the matching $2015-2019$ Polish MNiSW fund.

\section{Preliminaries}

\subsection{Simplicial complexes}

Let $X$ be a simplicial complex.
We denote by $X^{(k)}$ the $k$-skeleton of $X, 0 \le k < \dim X$.
A subcomplex $L$ in $X$ is called \emph{full} as a subcomplex of $X$ if any simplex of $X$ spanned by a set of vertices in $L$, is a simplex of $L$.
For a set $A = \{ v_{1}, ..., v_{k} \}$ of vertices of $X$, by $\langle A \rangle$ or by $\langle v_{1}, ..., v_{k} \rangle$ we denote the \emph{span} of $A$, i.e. the smallest full subcomplex of $X$ that contains $A$.
We write $v \sim v'$ if $\langle v,v' \rangle \in X$ (it can happen that $v = v'$).
We write $v \nsim v'$ if $\langle v,v' \rangle \notin X$.
We call $X$ {\it flag} if any finite set of vertices, which are pairwise connected by edges of $X$, spans a simplex of $X$.

A {\it cycle} ({\it loop}) $\gamma$ in $X$ is a subcomplex of $X$ isomorphic to a triangulation of $S^{1}$.
A \emph{full cycle} in $X$ is a cycle that is full as a subcomplex of $X$.
A $k$-\emph{wheel} in $X$ $(v_{0}; v_{1}, ..., v_{k})$ (where $v_{i}, i \in \{0,..., k\}$ are vertices of $X$) is a subcomplex of $X$ such that $(v_{1}, ..., v_{k})$ is a full cycle and $v_{0} \sim v_{1}, ..., v_{k}$.
The \emph{length} of $\gamma$ (denoted by $|\gamma|$) is the number of edges of $\gamma$.

We define the \emph{metric} on the $0$-skeleton of $X$ as the number of edges in the shortest $1$-skeleton path joining two given vertices.

Let $\sigma$ be a simplex of $X$.
The \emph{link} of $X$ at $\sigma$, denoted $X_{\sigma}$, is the subcomplex of $X$ consisting of all simplices of $X$ which are disjoint from $\sigma$ and which, together with $\sigma$, span a simplex of $X$.
We call a flag simplicial complex \emph{k-large} if there are no full $j$-cycles in $X$, for $j < k$.
We say $X$ is \emph{locally k-large} if all its links are $k$-large.
We call a vertex of $X$ \emph{k-large} if its link is $k$-large.

\begin{definition}\label{def:simplicial-map}
A \emph{simplicial map} $f : X \rightarrow Y$ between simplicial complexes $X$ and $Y$ is a map which sends vertices to vertices, and whenever vertices $v_{0}, ..., v_{k} \in X$ span a simplex $\sigma$ of $X$ then their images span a simplex $\tau$ of $Y$ and we have $f(\sigma) = \tau$.
Therefore a simplicial map is determined by its values on the vertex set of $X$.
A simplicial map is \emph{nondegenerate} if it is injective on each simplex.
\end{definition}

\begin{definition}\label{def:filling-diagram}
Let $\gamma$ be a cycle in $X$.
A \emph{filling diagram} for $\gamma$ is a simplicial map $f : D \rightarrow X$ where $D$ is a triangulated $2$-disc, and $f | _{\partial D}$ maps $\partial D$ isomorphically onto $\gamma$.
We denote a filling diagram for $\gamma$ by $(D,f)$ and we say it is
\begin{itemize}
    \item \emph{minimal} if $D$ has minimal area (it consists of the least possible number of $2$-simplices among filling diagrams for $\gamma$);
    \item \emph{nondegenerate} if $f$ is a nondegenerate map;
    \item \emph{locally $k$-large} if $D$ is a locally $k$-large simplicial complex.
\end{itemize}
\end{definition}

\begin{lemma}\label{lemma:filling-diagram}
Let $X$ be a simplicial complex and let $\gamma$ be a homotopically trivial loop in $X$.
Then:
\begin{enumerate}
    \item there exists a filling diagram $(D, f)$ for $\gamma$ (see \cite{Ch} - Lemma $5.1$, \cite{JS1} - Lemma $1.6$ and \cite{Pr} - Theorem $2.7$);
    \item any minimal filling diagram for $\gamma$ is simplicial and nondegenerate (see \cite{Ch} - Lemma $5.1$, \cite{JS1} - Lemma $1.6$, Lemma $1.7$ and \cite{Pr} - Theorem $2.7$).
\end{enumerate}
\end{lemma}

Let $D$ be a simplicial disc.
We denote by $C$ the cycle bounding $D$ and by $\rm{Area} C$ the area of $D$.
We denote by $V_{i}$ and $V_{b}$ the numbers of internal and boundary vertices of $D$, respectively.
Then: $\rm{Area} C = 2 V_{i} + V_{b} - 2 = |C| + 2 (V_{i} - 1)$ (Pick's formula).
In particular, the area of a simplicial disc depends only on the numbers of its internal and boundary vertices.

\begin{definition}\label{def:t-uniform}
Let $X$ be a flag, simply connected simplicial complex and let $\gamma$ be a loop in $X$.
We call $X$ \emph{t-uniform}, $t \ge 4$, if in any minimal filling diagram $(D,f)$ for $\gamma$, for any interior vertex $v$ of $D$, we have $|D_{v}| = t$ (i.e. any interior vertex $v$ of $D$ has $t$ neighbours).
\end{definition}

Let $X$ be a $t$-uniform simplicial complex and let $\gamma$ be a loop in $X$.
Let $(D,f)$ be a minimal filling diagram for $\gamma$ and let $v$ be a vertex of $D$.
We call the \textit{sphere} centered at $v$ of radius $n$ the set of edges spanned by the vertices at distance $n$ from $v$, $n \ge 0$.
We denote it by $S_n^t$.
We call the \textit{area} of a sphere the number of triangles inside the sphere.
We denote it by $A_{n}^{t}$.
We call the \textit{length} of a sphere the number of edges on the sphere.
We denote it by $|S_n^t|$.

If $Z$ is a set of vertices, we denote by $|Z|$ the number of its vertices.

\subsection{Lucas sequences}

\subsubsection{General considerations}

Given two integer parameters $P$ and $Q$, the Lucas sequences of the first kind $(U_n(P,Q))_{n \ge 0}$ and of the second kind $(V_n(P,Q))_{n \ge 0}$ are defined by the following recurrence relations (see \cite{Ri})
\begin{itemize}
    \item $U_0(P,Q)=0,$
    \item $U_1(P,Q)=1,$
    \item $U_n(P,Q) = P \cdot U_{n-1}(P,Q) - Q \cdot U_{n-2}(P,Q)$, for $n>1$
\end{itemize}
and
\begin{itemize}
    \item $V_0(P,Q)=2,$
    \item $V_1(P,Q)=P,$
    \item $V_n(P,Q) = P \cdot V_{n-1}(P,Q) - Q \cdot V_{n-2}(P,Q)$, for $n>1$.
\end{itemize}

The characteristic equation of the recurrence relation for the Lucas sequences \\ $U_n(P,Q)$ and $V_n(P,Q)$ is
$$x^2 - P \cdot x + Q = 0$$
It has the discriminant $D = P^2 - 4Q$ and the roots
$$a = \cfrac{P + \sqrt{P^2 - 4Q}}{2},\quad
b = \cfrac{P - \sqrt{P^2 - 4Q}}{2}$$

We discuss two cases.

$\bullet$ If $D\neq 0$, then $a$ and $b$ are distinct and we have
$$a^n = \cfrac{V_n + U_n \sqrt{D}}{2},\quad
b^n = \cfrac{V_n - U_n \sqrt{D}}{2}$$
Then the terms of the Lucas sequences can be expressed in terms of $a$ and $b$ as follows
\begin{equation}\label{eq:lucas-n-term-using-a-b}
    U_n = \cfrac{a^n - b^n}{a-b},\quad
    V_n = a^n + b^n
\end{equation}

$\bullet$ If $D = 0$, then $P = 2S$ and $Q = S^2$ for some integer $S$ so that $a = b = S$.
In this case we have
\begin{equation}\label{eq:lucas-n-term-using-s}
    U_n(P,Q) = U_n(2S, S^2) = nS^{n-1}
\end{equation}
$$V_n(P,Q) = V_n(2S, S^2) = 2S^n$$

\subsubsection{Fibonacci sequence}

The Fibonacci sequence $(F)_{n \ge 0}$ is a special case of Lucas sequence of the first kind for $P=1$, $Q=-1$: $F_n = U_n(1,-1)$.
The recurrence relation is
\begin{center}
    $F_n = F_{n-1} + F_{n-2}$
\end{center}

The initial terms of the Fibonacci sequence are $0, 1, 1, 2, 3, 5, 8, 13, 21, 34, 55, ...$

The limit of the ratio of two successive terms of the Fibonacci sequence is the \textit{golden ratio}
$$\lim_{n\to\infty} \cfrac{F_{n+1}}{F_n} = \varphi = \cfrac{1 + \sqrt{5}}{2}$$

\subsubsection{The bisection of the Fibonacci sequence}

The bisection of the Fibonacci sequence $(B)_{n \ge 0}$ contains the terms on even positions of the Fibonacci sequence.
It is the Lucas sequence of the first kind for $P=3$, $Q=1$: $B_n = U_n(3,1)$.
The recurrence relation is
$$B_n = 3 \cdot B_{n-1} - B_{n-2}$$

The initial terms of the bisection of the Fibonacci sequence are $0, 1, 3, 8, 21, 55, ...$

Expressed in terms of the Fibonacci sequence, we have $B_n = F_{2n}$.

\section{Lucas sequences in simplicial complexes}
 
The main goal of this section is to find, in terms of $t$, the best constant for the isoperimetric inequality for a $t$-uniform simplicial complex $X$, $t \ge 6$.
The only loops we consider are the spheres inside the disc of a minimal filling diagram associated to a loop of $X$.
Namely, we compute for $t \ge 7$ the limit of the ratio $\cfrac{A_n^t}{|S_n^t|}$ as $n$ goes to infinity (Theorem \ref{theorem:area-length-ratio-t-unif}).
For $t = 6$ we show that the ratio $\cfrac{A_n^t}{|S_n^t|^2}$ is constant (Theorem \ref{theorem:area-length-ratio-6-unif}).

In section \ref{sec:examples}, we give examples of $t$-uniform simplicial complexes.
In section \ref{sec:lucas-t-unif-connection} we find the relation between the lengths of spheres in $t$-uniform simplicial complexes and the terms of certain Lucas sequences (Theorem \ref{theorem:sphere-lucas-first-kind}).
Moreover, we express the area of a sphere in terms of the lengths of certain spheres (Theorem \ref{theorem:area-in-sphere}).
In section \ref{sec:bisection-fib} we study the bisection of the Fibonacci sequence in $7$-uniform simplicial complexes.
In section \ref{sec:lucas-seq-t-unif} we analyze Lucas sequences for $Q=1$ in $t$-uniform simplicial complexes, $t \ge 4$.
In section \ref{sec:tables-of-sequences} we present tables of sequences involving lengths and areas of spheres in $t$-uniform simplicial complexes, $4 \le t \le 10$.

\begin{figure}[ht]
    \centering

    \includegraphics[height=0.44\textheight]{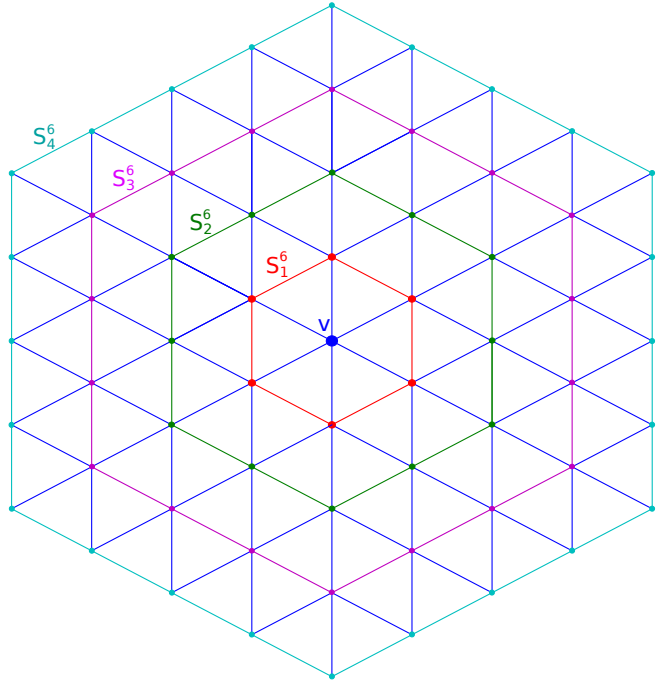}
    \caption{A minimal filling diagram of a $6$-uniform simplicial complex}
    \label{fig:all-6-levels-4}
\end{figure}

\subsection{Examples}\label{sec:examples}

We start by presenting a few examples of $t$-uniform simplicial complexes.
Namely,
\begin{itemize}
    \item for $t = 4$: an octahedron (Figure \ref{fig:octahedron});
    \item for $t = 5$: an icosahedron (Figure \ref{fig:icosahedron});
    \item for $t = 6$: the regular tessellation of the Euclidean plane by equilateral triangles (Figure \ref{fig:all-6-levels-4});
    \item for $t \ge 7$: the complex is hyperbolic (Figures \ref{fig:all-7-levels-5} and \ref{fig:all-8-levels-5});
    \begin{itemize}
        \item if all triangles are equilateral with an angle measuring $\pi/3$ at each vertex, then the sum of the measures of the angles around each vertex is $t\pi/3$, which is bigger than $2\pi$;
        \item if the sum of the measures of the angles around each vertex is equal to $2\pi$, then each triangle has angles of measure $2\pi/t$ at each vertex, which is less then $\pi/3$.
    \end{itemize}
\end{itemize}

\begin{figure}[p]
    \centering
    \includegraphics[height=0.44\textheight]{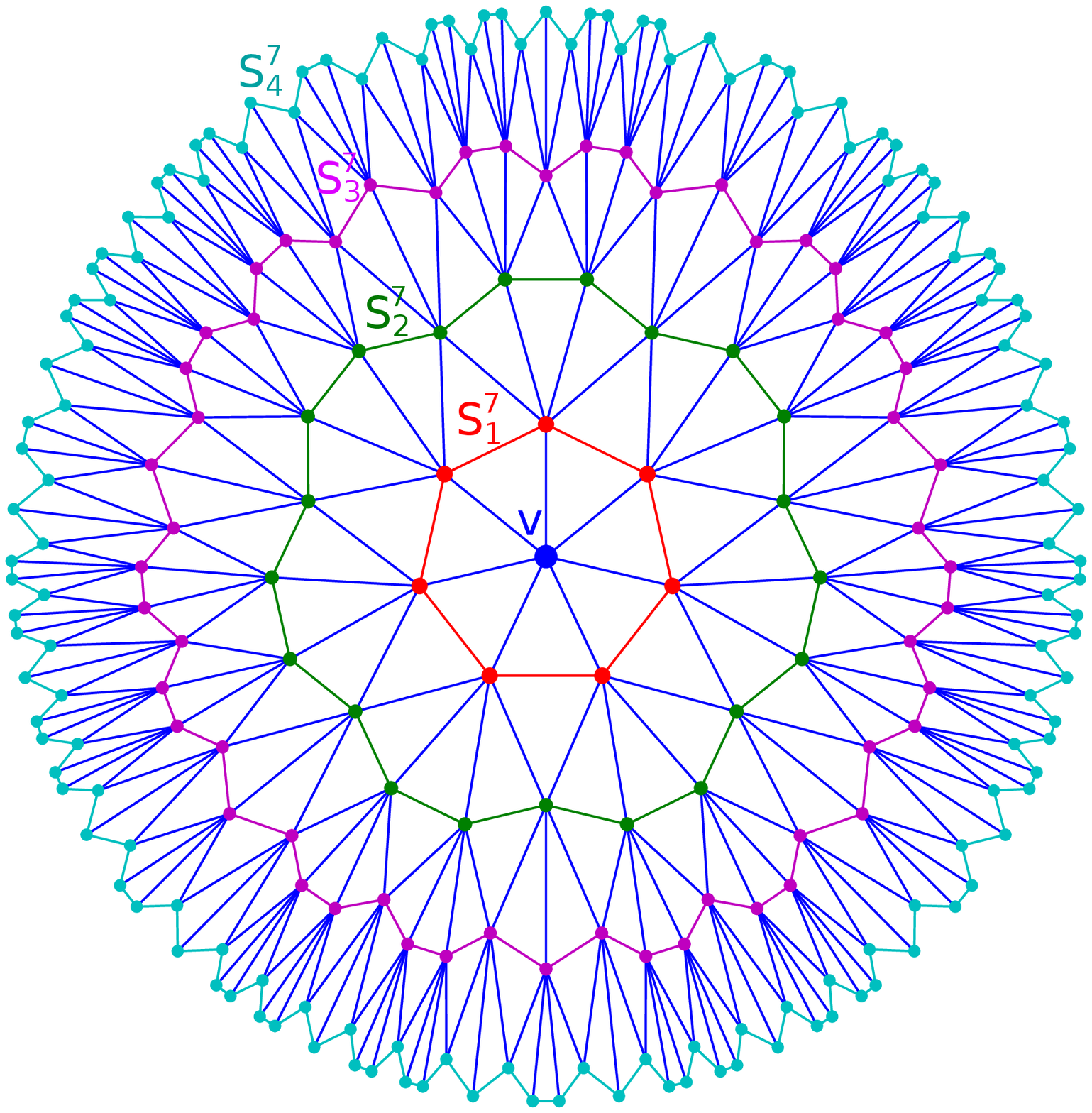}
    \caption{A minimal filling diagram of a $7$-uniform simplicial complex}
    \label{fig:all-7-levels-5}

    \vspace{10pt}

    \includegraphics[height=0.44\textheight]{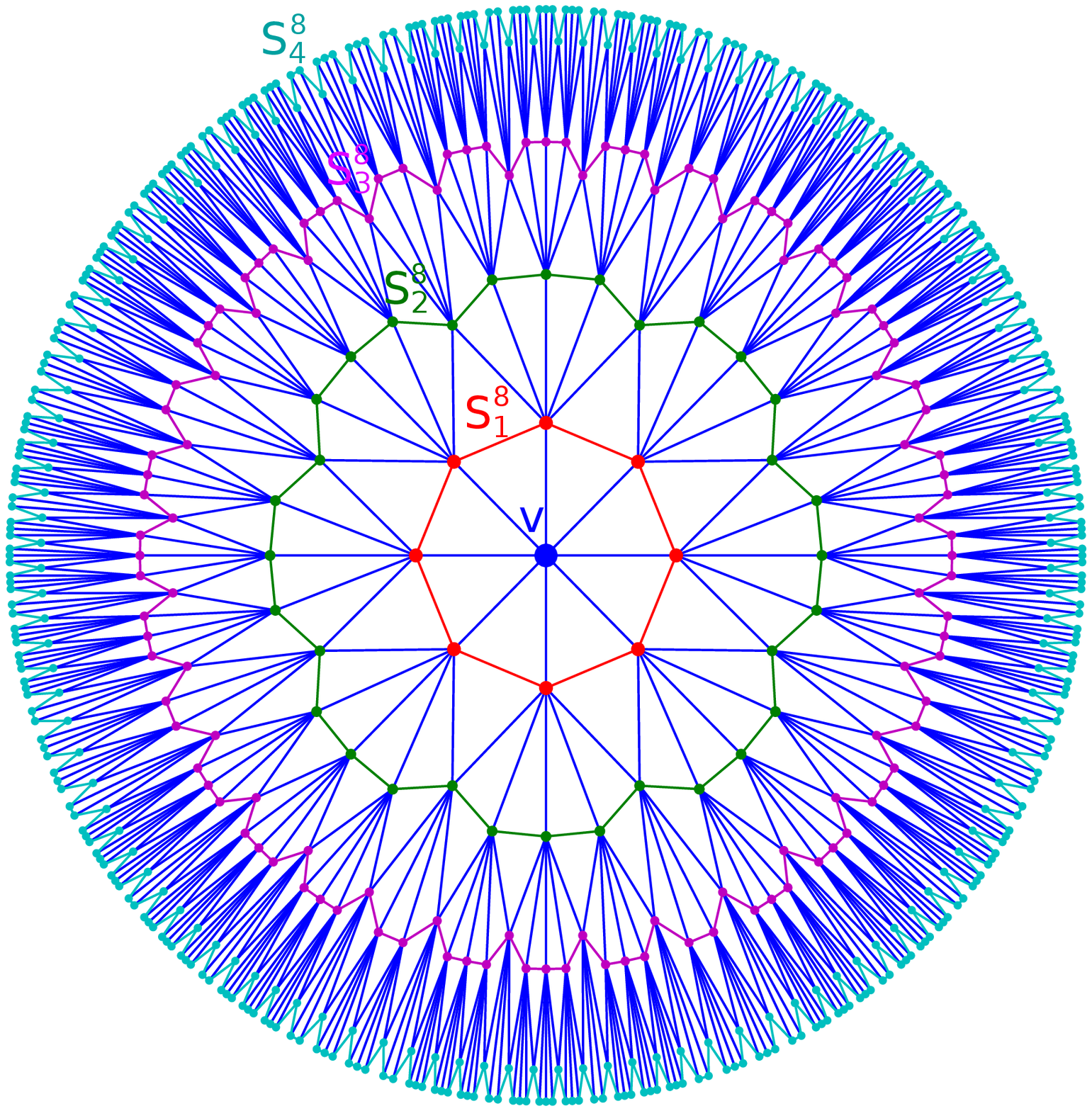}
    \caption{A minimal filling diagram of an $8$-uniform simplicial complex}
    \label{fig:all-8-levels-5}
\end{figure}

\subsection{Lucas sequences in t-uniform simplicial complexes}\label{sec:lucas-t-unif-connection}

We start by establishing, for $t \ge 6$, a connection between the lengths of spheres in $t$-uniform simplicial complexes and the terms of certain Lucas sequences.

\begin{theorem}\label{theorem:sphere-lucas-first-kind}
Let $t \ge 6$ and let $X$ be a $t$-uniform simplicial complex.
Let $\gamma$ be a loop of $X$ and let $(D,f)$ be a minimal filling diagram for $\gamma$.
Let $v$ be an interior vertex of $D$.
Then the lengths of the spheres centered at $v$ are Lucas sequences of the first kind $U_n$ with parameters $P = t-4$ and $Q = 1$ multiplied by $t$.
Namely,
\begin{equation}\label{eq:sphere-lucas-first-kind}
    |S_n^t| = t \cdot U_n(t-4, 1)
\end{equation}
\end{theorem}

\begin{proof}
We consider consecutive spheres of $D$ centered at the same vertex $v$.
We count the number of vertices on each sphere.
In general the length of a sphere equals the number of vertices on the sphere.
The sphere made of a single vertex, however, is an exception (i.e. the sphere $S_0$ at distance $0$).
Namely, although the length of the sphere is $0$, the number of its vertices is equal to $1$.
Let $n \ge 2$.
In order to count the number of vertices on the sphere $S_n^t$, we split the vertices on $S_{n-1}^t$ into two sets:
\begin{itemize}
    \item the set $Y_{n-1}^t$ contains those vertices connected to two interior vertices on $S_{n-2}^t$ which are adjacent,
    \item the set $Z_{n-1}^t$ contains those vertices connected to one interior vertex on $S_{n-2}^t$.
\end{itemize}

The vertices from the set $Y_{n-1}^t$ are connected to two interior vertices on $S_{n-2}^t$, to two vertices on the same sphere $S_{n-1}^t$, and, because the complex is $t$-uniform, to other $t - 4$ vertices on the exterior sphere $S_n^t$.
The number of these vertices is equal to the number of edges on the sphere $S_{n-2}^t$.
We note that each vertex in $Y_{n-1}^{t}$ corresponds to an edge on $S_{n-2}^{t}$.
So $|Y_{n-1}^t| = |S_{n-2}^t|$.
The vertices from the set $Z_{n-1}^t$ are connected to one interior vertex on $S_{n-2}^t$, to two vertices on $S_{n-1}^t$, and, because the complex is $t$-uniform, to other $t - 3$ vertices on the exterior sphere $S_n^t$.
Any two vertices spanning an edge on $S_{n-1}^t$ are connected to the same vertex on $S_n^t$.
Therefore, in order to obtain the number of vertices on $S_n^t$, we have to count one vertex less for each vertex on $S_{n-1}^t$.

In conclusion, for $n \ge 2$, the number of vertices on $S_n^t$ is equal to
$$|S_n^t| = [(t - 4) - 1] \cdot |Y_{n-1}^t| + [(t - 3) - 1] \cdot |Z_{n-1}^t| =$$
$$= (t - 5) \cdot |Y_{n-1}^t| + (t - 4) \cdot |Z_{n-1}^t| =$$
$$= (t - 4) \cdot (|Y_{n-1}^t| + |Z_{n-1}^t|) - |Y_{n-1}^t| =$$
$$= (t - 4) \cdot |S_{n-1}^t| - |S_{n-2}^t|$$

The initial terms are $|S_0^t| = 0$ and $|S_1^t| = t$.
This implies that the lengths of spheres in $t$-uniform simplicial complexes are Lucas sequences of the first kind $U_n$ with parameters $P = t-4$ and $Q = 1$ multiplied by $t$: $|S_n^t| = t \cdot U_n(t-4, 1)$.
\end{proof}

Next we express the area between two consecutive spheres in terms of the lengths of these spheres.
We denote by $A_{k-1,k}^t$ the area between two consecutive spheres $S_{k-1}^{t}$ and $S_{k}^{t}$ centered at $v$.

\begin{lemma}\label{lemma:area-2-spheres}
Let $X$ be a $t$-uniform simplicial complex and let $\gamma$ be a loop in $X$.
Let $(D, f)$ be a minimal filling diagram for $\gamma$.
Let $v$ be an interior vertex of $D$.
Then the area between two consecutive spheres $S_{k-1}^t$ and $S_k^t$ centered at $v$ is equal to the sum of the lengths of these spheres
\begin{equation}\label{eq:ring-area-sum-spheres}
    A_{k-1,k}^t = |S_{k-1}^t| + |S_k^t|
\end{equation}
\end{lemma}

\begin{figure}[ht]
    \centering
    \includegraphics[width=0.4\textwidth]{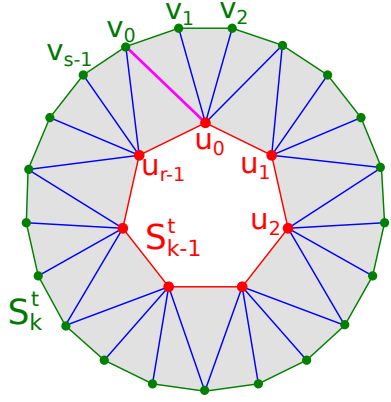}
    \caption{Area between two consecutive spheres centered at the same vertex}
    \label{fig:area-between-two-spheres}
\end{figure}

\begin{proof}
Assume that $S_{k-1}^t$ has $r$ vertices ($u_0$, $u_1$, ..., $u_{r-1}$), $S_k^t$ has $s$ vertices ($v_0$, $v_1$, ..., $v_{s-1}$), $u_0 \sim v_0$, and that the indices of the two sequences of vertices increase in the same direction.
We start counting the edges and the triangles between the spheres with the edge $\la u_0, v_0 \ra$ (see Figure \ref{fig:area-between-two-spheres}).
This edge is included either in the triangle $\la u_0, v_0, v_1 \ra$ which has an edge on $S_k^t$, or in the triangle $\la u_0, v_0, u_1 \ra$ which has an edge on $S_{k-1}^t$.
The figure below illustrates only the first case.
We continue the counting either with $\la u_0, v_1 \ra$ or with $\la u_1, v_0 \ra$.
This edge belongs to a triangle, which has an edge on $S_{k-1}^t$ or on $S_k^t$ and another edge between vertices on the two spheres.
The counting is complete once we return to the edge $\la u_0, v_0 \ra$.
For each edge joining vertices on both spheres $S_{k-1}^t$ and $S_k^t$, we count one triangle.
In conclusion the number of triangles between the spheres $S_{k-1}^t$ and $S_k^t$ is equal to the sum of the number of edges on both spheres.
The number of these triangles represents the area $A_{k-1,k}^t$.
\end{proof}

\begin{theorem}\label{theorem:area-in-sphere}
Let $X$ be a $t$-uniform simplicial complex and let $\gamma$ be a loop in $X$.
Let $(D, f)$ be a minimal filling diagram for $\gamma$.
Let $v$ be an interior vertex of $D$.
Then the area of a sphere centered at $v$ is equal to \begin{equation}\label{eq:area-sum-spheres}
    A_n^t = 2 \bigg(\sum_{k=0}^{n}{|S_k^t|}\bigg) - |S_n^t|
\end{equation}
\end{theorem}

\begin{proof}
The area of a sphere is equal to the sum of the areas between each pair of consecutive spheres around $v$.
Namely,
$$A_n^t = \sum_{k=1}^n{A_{k-1,k}^t}$$
Therefore, using (\ref{eq:ring-area-sum-spheres}), we can express this area in terms of the lengths of the spheres around $v$ as follows:
$$A_n^t = \sum_{k=1}^n{A_{k-1,k}^t} = \sum_{k=1}^n{(|S_{k-1}^t| + |S_k^t|)}
= 2 \bigg(\sum_{k=0}^{n}{|S_k^t|}\bigg) - |S_n^t|$$
\end{proof}

\subsection{The bisection of the Fibonacci sequence in 7-uniform simplicial complexes}\label{sec:bisection-fib}

As a particular case, relation (\ref{eq:sphere-lucas-first-kind}) implies that spheres in a $7$-uniform simplicial complex are the terms of the bisection of the Fibonacci sequence multiplied by $7$.
Namely,
\begin{equation}\label{eq:sphere-fib}
    |S_n^7| = 7 \cdot U_n(3,1) = 7 \cdot B_n = 7 \cdot F_{2n}
\end{equation}

The sum of the first $n+1$ elements from the bisection of the Fibonacci sequence is equal to
$$\sum_{k=0}^{n}{B_k} = \sum_{k=0}^{n}{F_{2k}} =$$
$$= F_0 + F_2 + F_4 + ... + F_{2n} =$$
$$= F_1 + F_2 + F_4 + ... + F_{2n} - F_1 =$$
$$= F_3 + F_4 + F_6 + ... +F_{2n} - 1 =$$
$$= F_5 + F_6 + ... + F_{2n} - 1 = ... =$$
\begin{equation}\label{eq:bisection-to-fib}
    = F_{2n+1}-1
\end{equation}

Then the relations (\ref{eq:sphere-fib}) and (\ref{eq:bisection-to-fib}) imply that
\begin{equation}\label{eq:sum-7-over-last}
    \lim_{n\to\infty} \cfrac{\sum_{k=0}^{n}|S_k^7|}{|S_n^7|} =
    \lim_{n\to\infty} \cfrac{7 \cdot \sum_{k=0}^{n}F_{2k}}{7 \cdot F_{2n}} =
    \lim_{n\to\infty} \cfrac{F_{2n+1}-1}{F_{2n}} = \varphi
\end{equation}

From (\ref{eq:area-sum-spheres}) and (\ref{eq:sum-7-over-last}), we get the limit, as $n$ goes to infinity, of the ratio between the area and the length of spheres in $7$-uniform simplicial complexes.
Namely,
$$\lim_{n\to\infty} \cfrac{A_n^7}{|S_n^7|} =
\lim_{n\to\infty} \bigg( 2 \cdot \cfrac{\sum_{k=0}^{n}|S_k^7|}{|S_n^7|} - \cfrac{|S_n^7|}{|S_n^7|} \bigg) =$$
\begin{equation}\label{eq:area-7-over-last}
    = 2 \varphi - 1 = 2 \cdot \cfrac{1 + \sqrt{5}}{2} - 1 = \sqrt{5}
\end{equation}

As shown below (Theorem \ref{theorem:area-length-ratio-t-unif}), the sequence $\bigg(\cfrac{A_n^7}{|S_n^7|}\bigg)_{n \ge 0}$ is strictly increasing.
Therefore for spheres in $7$-uniform simplicial complexes, the following inequality holds:
\begin{equation}\label{eq:isoper-7-uniform}
    A_n^7 < \sqrt{5} \cdot |S_n^7|
\end{equation}

As above we split the vertices on each sphere $S_n^7$ into two sets:
\begin{itemize}
    \item the set $Y_{n}^7$ contains the vertices connected to two interior vertices on $S_{n-1}^7$ which are adjacent,
    \item the set $Z_{n}^7$ contains the vertices connected to one interior vertex on $S_{n-1}^7$.
\end{itemize}

We note that:
\begin{itemize}
    \item $|Y_{1}^7|$ = $7 * 0$,
    \item $|Z_{1}^7|$ = $7 * 1$,
    \item $|Y_{2}^7|$ = $7 * 1$,
    \item $|Z_{2}^7|$ = $7 * 2$,
    \item $|Y_{3}^7|$ = $7 * 3$,
    \item $|Z_{3}^7|$ = $7 * 5$.
\end{itemize}
So it turns out that the sequence ($|Y_{1}^7|$, $|Z_{1}^7|$, $|Y_{2}^7|$, $|Z_{2}^7|$, $|Y_{3}^7|$, $|Z_{3}^7|$, ...) is the Fibonacci sequence multiplied by 7.

\subsection{Lucas sequences for Q=1 in t-uniform simplicial complexes.}\label{sec:lucas-seq-t-unif}

Using Theorem \ref{theorem:sphere-lucas-first-kind}, we continue studying Lucas sequences of the first kind $U_n(P,1)$.
The recurrence relation is
$$U_n(P,1) = P \cdot U_{n-1}(P,1) - U_{n-2}(P,1)$$
The characteristic equation of the recurrence relation is
$$x^2 - P \cdot x + 1 = 0$$
It has the discriminant $D = P^2 - 4$ and the roots
\begin{equation}\label{eq:lucas-q-1-roots}
    a = \cfrac{P + \sqrt{P^2 - 4}}{2},\quad
    b = \cfrac{P - \sqrt{P^2 - 4}}{2}
\end{equation}

Thus $ab = 1$, and we have $b = \cfrac{1}{a}$.

We discuss two cases: either $P \neq 2$, or $P = 2$.

$\bullet$ If $P \neq 2$ then $D \neq 0$. Note that the roots $a$ and $b$ are distinct.
In this case the terms of the Lucas sequence of the first kind given in (\ref{eq:lucas-n-term-using-a-b}) can be expressed only in terms of $a$.
Namely,
$$U_n(P,1) = \cfrac{a^n - \cfrac{1}{a^n}}{a - \cfrac{1}{a}} =$$
$$= \cfrac{a}{a^2-1} \cdot \cfrac{(a^n-1)(a^n+1)}{a^n}$$

If $P > 2$ then $D > 0$.
In this case $a$ and $b$ are distinct real numbers and $a > 1 > b$.

If $P < 2$ then $D < 0$.
In this case $a$ and $b$ are distinct complex numbers.

For $P > 2$ the sum of the first $n+1$ elements of the Lucas sequence is
$$\sum_{k=0}^n{U_k}(P,1)
= \sum_{k=0}^n{\cfrac{a}{a^2-1} \cdot \bigg(a^k - \cfrac{1}{a^k} \bigg)} =$$
$$= \cfrac{a}{a^2-1} \cdot \bigg(\sum_{k=0}^n{a^k} - \sum_{k=0}^n{\cfrac{1}{a^k}}\bigg) =$$
$$= \cfrac{a}{a^2-1} \cdot \Bigg(\cfrac{a^{n+1}-1}{a-1} - \cfrac{1-\cfrac{1}{a^{n+1}}}{1-\cfrac{1}{a}}\Bigg) =$$
$$= \cfrac{a}{a^2-1} \cdot \bigg(\cfrac{a^{n+1}-1}{a-1} - \cfrac{a^{n+1}-1}{(a-1)a^n}\bigg) =$$
$$= \cfrac{a}{a^2-1} \cdot \cfrac{(a^{n+1}-1)(a^n-1)}{(a-1)a^n}$$

Moreover, the sum of the first $n+1$ elements of the Lucas sequence divided by the last element is
$$\cfrac{\sum_{k=0}^n{U_k}(P,1)}{U_n(P,1)} =
\cfrac{a}{a^2-1} \cdot \cfrac{(a^{n+1}-1)(a^n-1)}{(a-1)a^n} \cdot \cfrac{a^2-1}{a} \cdot \cfrac{a^n}{(a^n-1)(a^n+1)} =$$
\begin{equation}\label{eq:sum-lucas-over-last}
    = \cfrac{a^{n+1}-1}{(a-1)(a^n+1)}
\end{equation}

Since $a > 1$, the limit of the above expression when $n$ goes to infinity is
$$\lim_{n\to\infty} \cfrac{\sum_{k=0}^n{U_k(P,1)}}{U_n(P,1)} = \lim_{n\to\infty} \cfrac{a^{n+1}-1}{(a-1)(a^n+1)} =$$
$$= \lim_{n\to\infty} \cfrac{a^{n+1}-1}{a^{n+1} - a^n +a - 1} =$$
$$= \lim_{n\to\infty} \cfrac{1-\cfrac{1}{a^{n+1}}}{1 - \cfrac{1}{a} + \cfrac{1}{a^n} - \cfrac{1}{a^{n+1}}} =$$
$$= \cfrac{1}{1-\cfrac{1}{a}} = \cfrac{a}{a-1}$$ 

Using (\ref{eq:lucas-q-1-roots}), we can express the above result in terms of $P$ as follows:
$$\lim_{n\to\infty} \cfrac{\sum_{k=0}^n{U_k(P,1)}}{U_n(P,1)} =
\cfrac{\cfrac{P + \sqrt{P^2 - 4}}{2}}{\cfrac{P + \sqrt{P^2 - 4}}{2} - 1}
= \cfrac{P + \sqrt{P^2 - 4}}{P -2 + \sqrt{P^2 - 4}} =$$
$$= \cfrac{(P - 2 - \sqrt{P^2 - 4}) \cdot (P + \sqrt{P^2 - 4})}{(P - 2)^2 - \sqrt{P^2 - 4}^2} =$$
$$= \cfrac{P(P - 2) + (P-2)\sqrt{P^2 - 4} - P\sqrt{P^2 - 4} - \sqrt{P^2 - 4}^2}{P^2 - 4P + 4 - (P^2 - 4)} =$$
$$= \cfrac{P^2 - 2P - 2\sqrt{P^2 - 4} - P^2 + 4}{-4P + 8} =$$
$$= \cfrac{- 2P + 4 - 2\sqrt{P^2 - 4}}{-4(P - 2)} =$$
$$= \cfrac{P - 2 + \sqrt{(P-2)(P+2)}}{2(P - 2)} =$$
\begin{equation}\label{eq:lucas-lim-sum-over-last}
    = \cfrac{1 + \sqrt{\cfrac{P+2}{P-2}}}{2}
\end{equation}

$\bullet$ If $P = 2$ then $D = 0$. Note that $a = b = 1$.
Hence, $S = \cfrac{P}{2} = 1$.
Thus, using (\ref{eq:lucas-n-term-using-s}), we get
\begin{equation}\label{eq:lucas-2-1-n-term}
    U_n(2,1) = n S^{n-1}= n
\end{equation}

Then the sum of the first $n+1$ elements of the Lucas sequence is
\begin{equation}\label{eq:lucas-2-1-sum-over-n-term}
    \sum_{k=0}^n{U_k}(2,1) = \sum_{k=0}^n{k} = \cfrac{n(n+1)}{2}
\end{equation}

\begin{theorem}\label{theorem:area-length-ratio-t-unif}
In $t$-uniform simplicial complexes, $t \ge 7$, the following inequality holds:
\begin{equation}\label{eq:isoper-t-uniform}
    A_n^t < \sqrt{\cfrac{t-2}{t-6}} \cdot |S_n^t|
\end{equation}
In particular, the sequence $\bigg(\cfrac{A_{n}^{t}}{|S_{n}^{t}|}\bigg)_{n \ge 0}$ is strictly increasing.
\end{theorem}

\begin{proof}
Based on (\ref{eq:sphere-lucas-first-kind}), the lengths of spheres in a $t$-uniform simplicial complex are Lucas sequences of parameters $P = t-4$ and $Q = 1$ multiplied by $t$.
Therefore, using (\ref{eq:lucas-lim-sum-over-last}) we get the following
\begin{equation}\label{eq:spaheres-ratio-sum-last}
    \lim_{n\to\infty} \cfrac{\sum_{k=0}^n{|S_k^t|}}{|S_n^t|} =
    \lim_{n\to\infty} \cfrac{\sum_{k=0}^n{\big[ t \cdot U_k(t-4,1) \big] }}{t \cdot U_n(t-4,1)} =
    \cfrac{1 + \sqrt{\cfrac{t-2}{t-6}}}{2}
\end{equation}

Using (\ref{eq:area-sum-spheres}) and (\ref{eq:spaheres-ratio-sum-last}), for a sphere $S_{n}^{t}$ we get the limit, when $n$ goes to infinity, of the ratio between its area and its length.
Namely,
$$\lim_{n\to\infty} \cfrac{A_n^t}{|S_n^t|} =
\lim_{n\to\infty} \cfrac{2 \big(\sum_{k=0}^{n}{|S_k^t|}\big) - |S_n^t|}{|S_n^t|} =$$
$$= \lim_{n\to\infty} 2 \cdot \cfrac{\sum_{k=0}^{n}{|S_k^t|}}{|S_n^t|} - 1 =$$
$$= 2 \cdot \cfrac{1 + \sqrt{\cfrac{t-2}{t-6}}}{2} - 1
=\sqrt{\cfrac{t-2}{t-6}}$$

Also from (\ref{eq:sphere-lucas-first-kind}), (\ref{eq:area-sum-spheres}) and (\ref{eq:sum-lucas-over-last}), for $n > 0$, we have
$$\cfrac{A_n^t}{|S_n^t|} - \cfrac{A_{n-1}^t}{|S_{n-1}^t|}
= \cfrac{2 \big(\sum_{k=0}^{n}{|S_k^t|}\big) - |S_n^t|}{|S_n^t|}
- \cfrac{2 \big(\sum_{k=0}^{{n-1}}{|S_k^t|}\big) - |S_{n-1}^t|}{|S_{n-1}^t|} =$$
$$= 2 \bigg[\cfrac{7 \cdot \sum_{k=0}^{n}{U_k(t-4,1)}}{7 \cdot U_n(t-4,1)} - \cfrac{7 \cdot \sum_{k=0}^{n-1}{U_k(t-4,1)}}{7 \cdot U_{n-1}(t-4,1)}\bigg] =$$
$$= 2 \bigg[\cfrac{a^{n+1}-1}{(a-1)(a^n+1)}
- \cfrac{a^n-1}{(a-1)(a^{n-1}+1)}\bigg] =$$
$$= \cfrac{2}{a-1} \cdot \cfrac{a^{2n} + a^{n+1} - a^{n-1} - 1 - a^{2n} + 1}
{(a^n + 1)(a^{n-1} + 1)}$$
$$= \cfrac{2(a^{n+1} - a^{n-1})}{(a-1)(a^n + 1)(a^{n-1} + 1)}$$

As $a > 1$, it follows that $\cfrac{A_n^t}{|S_n^t|} - \cfrac{A_{n-1}^t}{|S_{n-1}^t|} > 0$.
So the sequence $\bigg(\cfrac{A_n^t}{|S_n^t|}\bigg)_{n \ge 0}$ is strictly increasing.

So for $t \ge 7$ in $t$-uniform simplicial complexes, we have
$A_n^t < \sqrt{\cfrac{t-2}{t-6}} \cdot |S_n^t|$.
\end{proof}

We note that for $7$-uniform simplicial complexes, Theorem \ref{theorem:area-length-ratio-t-unif} ensures that
relation (\ref{eq:isoper-7-uniform}) is indeed fulfilled:
$A_n^7 < \sqrt{5} \cdot |S_n^7|$.

Pick's formula implies that $A_n^t = |S_n^t| + 2 \cdot (V_{i} - 1)$.
We have denoted by $V_{i}$ the number of interior vertices of the disc enclosed by $S_{n}^{t}$.
Because $V_{i} \ge 1$, for $n \ge 1$, we have $A_n^t \ge |S_n^t|$.
Thus $\cfrac{A_n^t}{|S_n^t|} \ge 1$.
When $t$ goes to infinity, we get
$$\lim_{t\to\infty} \lim_{n\to\infty} \cfrac{A_n^t}{|S_n^t|} =
\lim_{t\to\infty} \sqrt{\cfrac{t-2}{t-6}} = 1$$

\begin{theorem}\label{theorem:area-length-ratio-6-unif}
In $6$-uniform simplicial complexes the following equality holds:
\begin{equation}\label{eq:isoper-6-uniform}
    A_n^6 = \cfrac{|S_n^6|^2}{6}
\end{equation}
\end{theorem}

\begin{proof}
Theorem \ref{theorem:sphere-lucas-first-kind} implies that $|S_n^6| = 6 \cdot U_n(2, 1)$.
Then using (\ref{eq:area-sum-spheres}), (\ref{eq:lucas-2-1-n-term}) and (\ref{eq:lucas-2-1-sum-over-n-term}), we get
$$\cfrac{A_n^6}{|S_n^6|^2}
= \cfrac{2 \big(\sum_{k=0}^{n}{|S_k^6|}\big) - |S_n^6|}{|S_n^6|^2}
= \cfrac{2 \cdot 6 \cdot \big[\sum_{k=0}^{n}{U_k(2,1)}\big]
- 6 \cdot U_n(2,1)}{6^2 \cdot [U_n(2,1)]^2} =$$
$$= \cfrac{2 \cdot \cfrac{n(n+1)}{2} - n}{6 \cdot n^2} = \cfrac{1}{6}$$
So we have $A_n^6 = \cfrac{|S_n^6|^2}{6}$.
\end{proof}

We present a few conjectures.

\begin{conjecture}\label{conjecture-k-large}
For $k \ge 7$, $k$-large simplicial complexes satisfy a linear isoperimetric inequality.
Namely,
$$A < \sqrt{\cfrac{k-2}{k-6}} \cdot L$$
\end{conjecture}

\begin{conjecture}\label{conjecture-7-located-k-large}
$6$-large simplicial complexes satisfy a quadratic isoperimetric inequality.
Namely,
$$A \le \cfrac{L^2}{6}$$
\end{conjecture}

We note that for $t = 4$ and $t = 5$, relation (\ref{eq:sphere-lucas-first-kind}) holds only for the initial values.
Namely, for $t = 4$ (Figure \ref{fig:octahedron}) let the octahedron be an example of $4$-uniform simplicial complex.
The largest possible sphere in a minimal filling diagram for a loop in the complex is $S_1^4$.
Each of the four vertices on $S_1^4$ are connected to $3$ vertices (one is the central vertex; the other two are on the same sphere $S_1^4$).
So because the complex is $4$-uniform, these vertices are connected to a single vertex on $S_{2}^{4}$.
This implies that $S_2^4$ contains a single vertex.
Such a situation is not possible in a minimal filling diagram because a disc is flat.
 
\begin{figure}[ht]
    \centering
    \begin{subfigure}[b]{.5\textwidth}
        \centering
        \includegraphics[width=.45\textwidth]{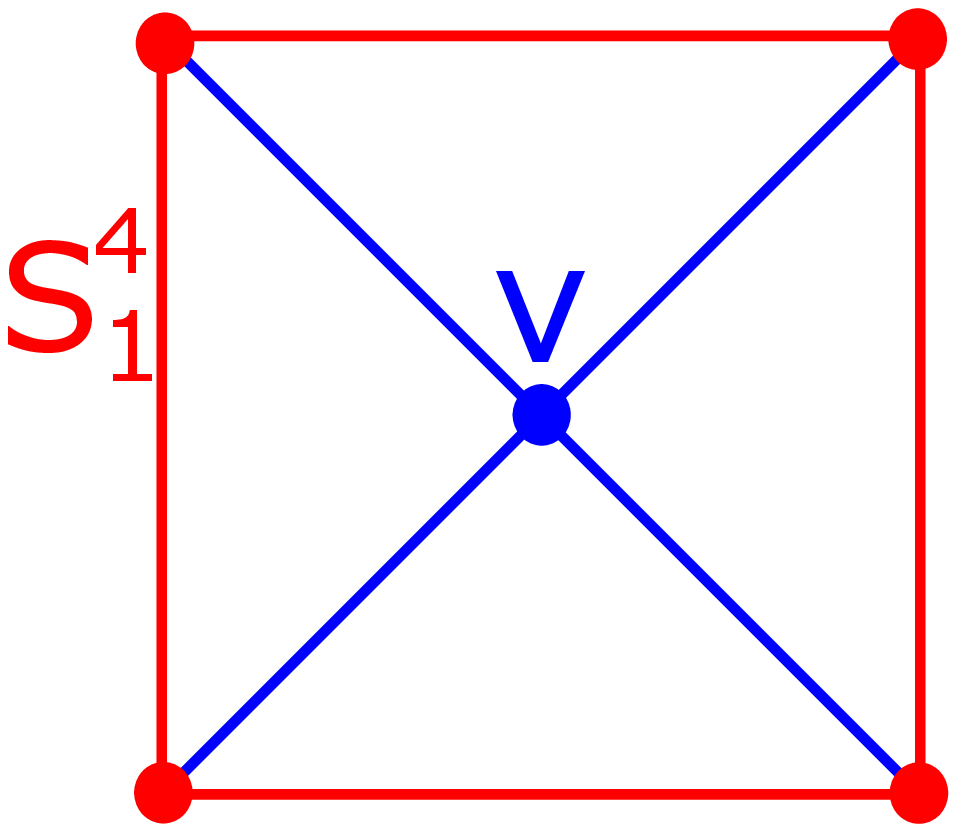}
        \caption{A minimal filling diagram}
    \end{subfigure}%
    \begin{subfigure}[b]{.5\textwidth}
        \centering
        \includegraphics[width=.7\textwidth]{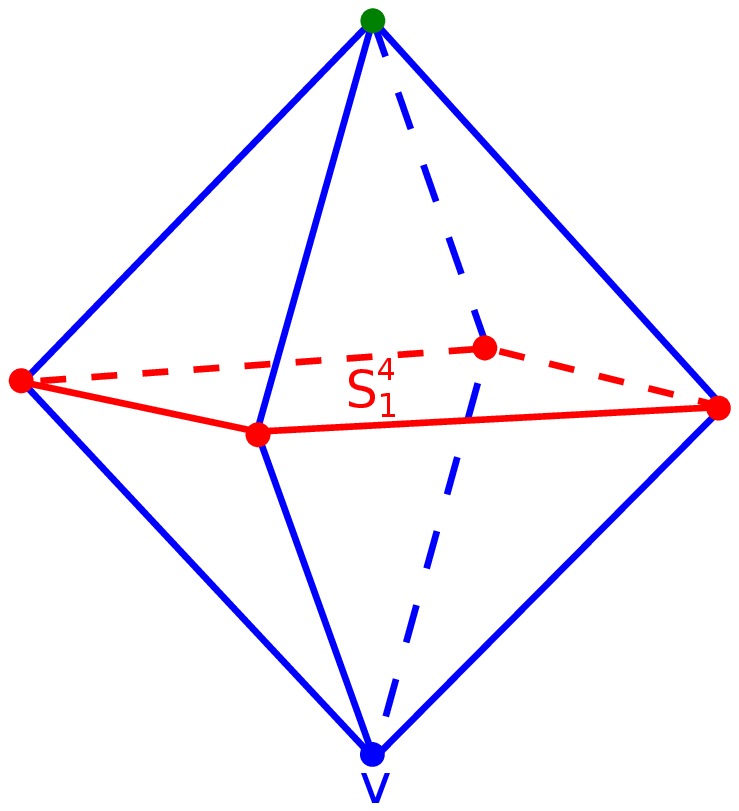}
        \caption{The octahedron}
    \end{subfigure}%
    \caption{A $4$-uniform simplicial complex}
    \label{fig:octahedron}
\end{figure}

One can reason similarly for $t = 5$ (Figure \ref{fig:icosahedron}).
Let the icosahedron with a missing vertex be an example of $5$-uniform simplicial complex.
The largest possible sphere in a minimal filling diagram for a loop in the complex is $S_2^5$.
Each of the five vertices on $S_1^5$ are connected to three vertices (one is the central vertex; the other two are on the same sphere $S_1^5$).
Because the complex is $5$-uniform, each of these vertices is connected to two of the five vertices on $S_2^5$.
The vertices on $S_2^5$ are connected to four vertices (two are on $S_1^5$; the other two are on $S_2^5$).
Because the complex is $5$-uniform, these vertices are connected to a single vertex on $S_3^5$.
This implies that the sphere $S_3^5$ contains a single vertex.
This is not possible in a minimal filling diagram because a disc is flat.

\begin{figure}[ht]
    \centering
    \begin{subfigure}[b]{.5\textwidth}
        \centering
        \includegraphics[width=.7\textwidth]{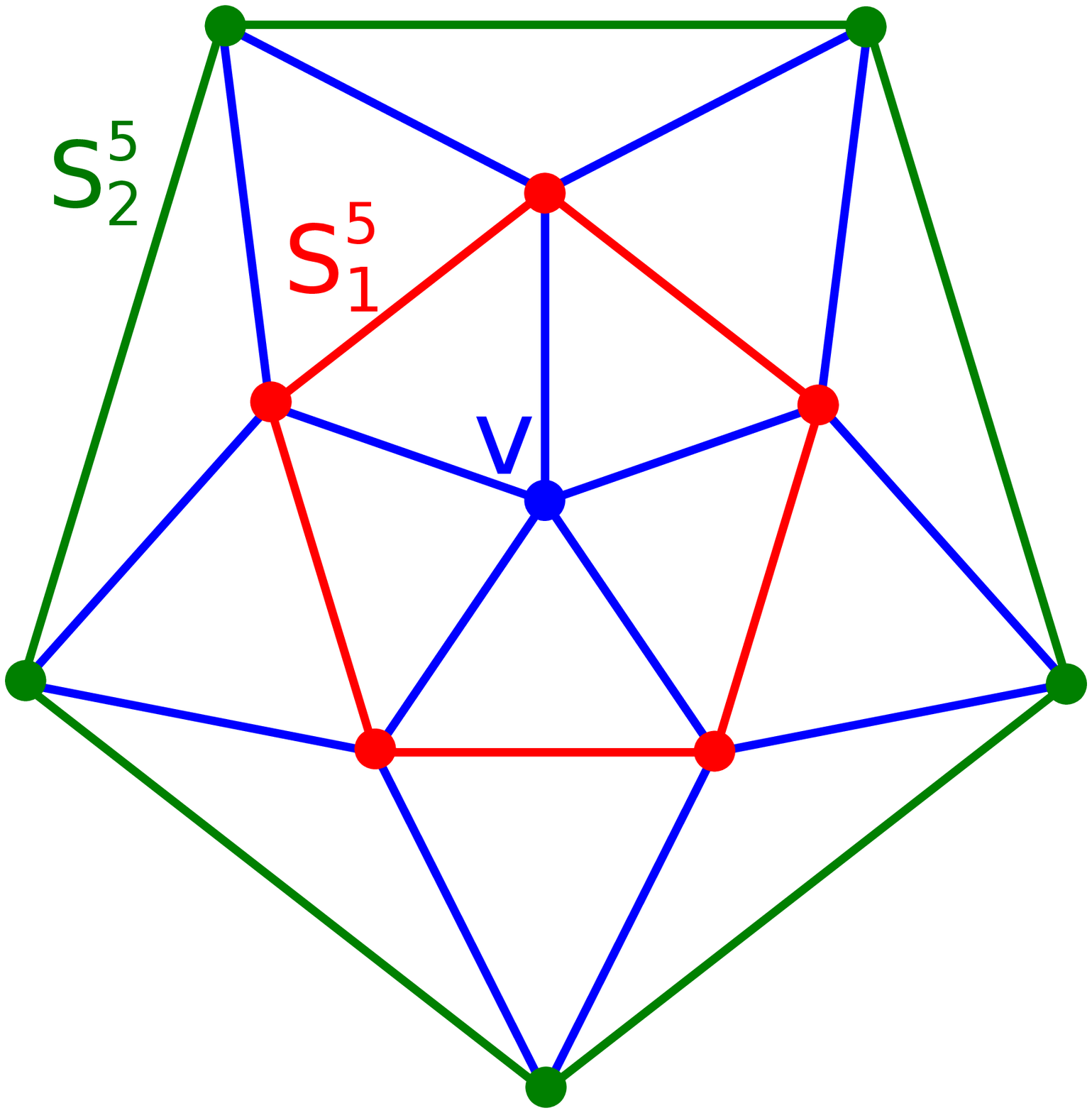}
        \caption{A minimal filling diagram}
    \end{subfigure}%
    \begin{subfigure}[b]{.5\textwidth}
        \centering
        \includegraphics[width=.7\textwidth]{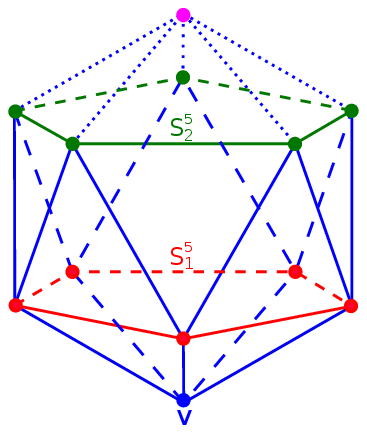}
        \caption{The icosahedron (with a missing vertex)}
    \end{subfigure}%
    \caption{A $5$-uniform simplicial complex}
    \label{fig:icosahedron}
\end{figure}

\subsection{Tables of sequences}\label{sec:tables-of-sequences}

We end by computing sequences of lengths of spheres, sum of lengths of spheres, areas between spheres and areas inside spheres in $t$-uniform simplicial complexes.
The values are divided by $t$ to outline the relation between these sequences and other integer sequences.
We introduce a table line to include approximations of the ratio between the area and the length of spheres.
For $t = 6$, an extra line represents the quadratic equality between the area and the length of spheres.
In the last table column we give links to some related entries in the \href{https://oeis.org}{On-Line Encyclopedia of Integer Sequences} (OEIS).

We start by presenting the tables for $6 \le t \le 10$.

\renewcommand{\arraystretch}{2}
\begin{table}[ht]
    \centering
    \begin{tabular}{|r l||r|r|r|r|r|r|r|r|r|r|}
         \hline
         n                             &      & 0 & 1 & 2 & 3 &  4 &  5 &  6 &  7 &  8 \\
         \hline \hline
         $|S_n^6|$                     & $/6$ & 0 & 1 & 2 & 3 &  4 &  5 &  6 &  7 &  8 \\
         \hline
         $\sum_{k=0}^{n} |S_k^6|$      & $/6$ & 0 & 1 & 3 & 6 & 10 & 15 & 21 & 28 & 36 \\
         \hline
         $A_{n-1,n}^6$                 & $/6$ & - & 1 & 3 & 5 &  7 &  9 & 11 & 13 & 15 \\
         \hline
         $A_n^6$                       & $/6$ & 0 & 1 & 4 & 9 & 16 & 25 & 36 & 49 & 64 \\
         \hline
         $\cfrac{A_n^6}{|S_n^6|}$      &      & - & 1 & 2 & 3 &  4 &  5 &  6 &  7 &  8 \\
         \hline
         $\cfrac{6 \cdot A_n^6}{|S_n^6|^2}$ & & - & 1 & 1 & 1 &  1 &  1 &  1 &  1 &  1 \\
         \hline
    \end{tabular}
    \caption{Sequences in a $6$-uniform simplicial complex}
    \label{tab:6-uniform}

    \begin{tabular}{|r l||r|r|r|r|r|r|r|r|r||r|}
         \hline
         n                         &      & 0 & 1 &     2 &  3 &    4 &   5 &     6 &    7 &     8 & 
         OEIS \\
         \hline \hline
         $|S_n^7|$                 & $/7$ & 0 & 1 &     3 &  8 &   21 &  55 &   144 &  377 &   987 &
         \href{https://oeis.org/A001906}{A001906} \\
         \hline
         $\sum_{k=0}^{n} |S_k^7|$  & $/7$ & 0 & 1 &     4 & 12 &   33 &  88 &   232 &  609 &  1596 &
         \href{https://oeis.org/A027941}{A027941} \\
         \hline
         $A_{n-1,n}^7$             & $/7$ & - & 1 &     4 & 11 &   29 &  76 &   199 &  521 &  1364 &
         \href{https://oeis.org/A002878}{A002878} \\
         \hline
         $A_n^7$                   & $/7$ & 0 & 1 &     5 & 16 &   45 & 121 &   320 &  841 &  2205 &
         \href{https://oeis.org/A004146}{A004146} \\
         \hline
         $\cfrac{A_n^7}{|S_n^7|}$  &      & - & 1 & 1.(6) &  2 & 2.14 & 2.2 & 2.(2) & 2.23 & 2.234 & 
         $\rightarrow \sqrt{5}$ \\
         \hline
    \end{tabular}
    \caption{Sequences in a $7$-uniform simplicial complex}
    \label{tab:7-uniform}

    \begin{tabular}{|r l||r|r|r|r|r|r|r|r||r|}
         \hline
         n                         &      & 0 & 1 &   2 &     3 &    4 &    5 &    6 &     7 &
         OEIS \\
         \hline \hline
         $|S_n^8|$                 & $/8$ & 0 & 1 &   4 &    15 &   56 &  209 &  780 &  2911 &
         \href{https://oeis.org/A001353}{A001353} \\
         \hline
         $\sum_{k=0}^{n} |S_k^8|$  & $/8$ & 0 & 1 &   5 &    20 &   76 &  285 & 1065 &  3976 &
         \href{https://oeis.org/A061278}{A061278} \\
         \hline
         $A_{n-1,n}^8$             & $/8$ & - & 1 &   5 &    19 &   71 &  265 &  989 &  3691 &
         \href{https://oeis.org/A001834}{A001834} \\
         \hline
         $A_n^8$                   & $/8$ & 0 & 1 &   6 &    25 &   96 &  361 & 1350 &  5041 &
         \href{https://oeis.org/A092184}{A092184} \\
         \hline
         $\cfrac{A_n^8}{|S_n^8|}$  &      & - & 1 & 1.5 & 1.(6) & 1.71 & 1.72 & 1.73 & 1.731 &
         $\rightarrow \sqrt{3}$ \\
         \hline
    \end{tabular}
    \caption{Sequences in an $8$-uniform simplicial complex}
    \label{tab:8-uniform}
\end{table}
\renewcommand{\arraystretch}{1}

\clearpage

\renewcommand{\arraystretch}{2}
\begin{table}[ht]
    \centering
    \begin{tabular}{|r l||r|r|r|r|r|r|r||r|}
         \hline
         n                         &      & 0 & 1 &   2 &   3 &    4 &     5 &     6 &
         OEIS \\
         \hline \hline
         $|S_n^9|$                 & $/9$ & 0 & 1 &   5 &  24 &  115 &   551 &  2640 &
         \href{https://oeis.org/A004254}{A004254} \\
         \hline
         $\sum_{k=0}^{n} |S_k^9|$  & $/9$ & 0 & 1 &   6 &  30 &  145 &   696 &  3336 &
         \href{https://oeis.org/A089817}{A089817} \\
         \hline
         $A_{n-1,n}^9$             & $/9$ & - & 1 &   6 &  29 &  139 &   666 &  3191 &
         \href{https://oeis.org/A030221}{A030221} \\
         \hline
         $A_n^9$                   & $/9$ & 0 & 1 &   7 &  36 &  175 &   841 &  4032 &
         \href{https://oeis.org/A054493}{A054493} \\
         \hline
         $\cfrac{A_n^9}{|S_n^9|}$  &      & - & 1 & 1.4 & 1.5 & 1.52 & 1.526 & 1.527 &
         $\rightarrow \sqrt{\cfrac{7}{3}}$ \\
         \hline
    \end{tabular}
    \caption{Sequences in a $9$-uniform simplicial complex}
    \label{tab:9-uniform}

    \begin{tabular}{|r l||r|r|r|r|r|r||r|}
         \hline
         n                              &       & 0 & 1 &     2 &   3 &    4 &     5 &
         OEIS \\
         \hline \hline
         $|S_n^{10}|$                   & $/10$ & 0 & 1 &     6 &  35 &  204 &  1189 &
         \href{https://oeis.org/A001109}{A001109} \\
         \hline
         $\sum_{k=0}^{n} |S_k^{10}|$    & $/10$ & 0 & 1 &     7 &  42 &  246 &  1435 &
         \href{https://oeis.org/A053142}{A053142} \\
         \hline
         $A_{n-1,n}^{10}$               & $/10$ & - & 1 &     7 &  41 &  239 &  1393 &
         \href{https://oeis.org/A002315}{A002315} \\
         \hline
         $A_n^{10}$                     & $/10$ & 0 & 1 &     8 &  49 &  288 &  1681 &
         \href{https://oeis.org/A001108}{A001108} \\
         \hline
         $\cfrac{A_n^{10}}{|S_n^{10}|}$ &       & - & 1 & 1.(3) & 1.4 & 1.41 & 1.413 &
         $\rightarrow \sqrt{2}$ \\
         \hline
    \end{tabular}
    \caption{Sequences in a $10$-uniform simplicial complex}
    \label{tab:10-uniform}
\end{table}
\renewcommand{\arraystretch}{1}

Next we present the tables for $t = 4$ and $t = 5$.
For $t = 4$ the only valid values for $n$ are $0$ and $1$.
For $t = 5$ the only valid values for $n$ are $0$, $1$ and $2$.
Larger values for $n$ are not valid.
Still we present them to see what values we would get in case such spheres would exist.
We note that the lengths of spheres would also be the terms of certain Lucas sequences.

\clearpage

\renewcommand{\arraystretch}{1.8}
\begin{table}[ht]
    \centering
    \begin{tabular}{|r l||r|r||d|d|d|d|d|d|d|}
         \hline
         n                          &      & 0 & 1 &        2 &  3 &  4 & 5 &        6 &  7 &  8 \\
         \hline \hline
         $|S_n^4|$                  & $/4$ & 0 & 1 &        0 & -1 &  0 & 1 &        0 & -1 &  0 \\
         \hline
         $\sum_{k=0}^{n} |S_k^4|$   & $/4$ & 0 & 1 &        1 &  0 &  0 & 1 &        1 &  0 &  0 \\
         \hline
         $A_{n-1,n}^4$              & $/4$ & - & 1 &        1 & -1 & -1 & 1 &        1 & -1 & -1 \\
         \hline
         $A_n^4$                    & $/4$ & 0 & 1 &        2 &  1 &  0 & 1 &        2 &  1 &  0 \\
         \hline
         $\cfrac{A_n^4}{|S_n^4|}$   &      & - & 1 & $\infty$ & -1 &  - & 1 & $\infty$ & -1 &  - \\
         \hline
    \end{tabular}
    \caption{Sequences in a $4$-uniform simplicial complex}
    \label{tab:4-uniform}

    \begin{tabular}{|r l||r|r|r||d|d|d|d|d|d|d|}
         \hline
         n                         &      & 0 & 1 & 2 &        3 &  4 &  5 &      6 & 7 & 8 \\
         \hline \hline
         $|S_n^5|$                 & $/5$ & 0 & 1 & 1 &        0 & -1 & -1 &      0 & 1 & 1 \\
         \hline
         $\sum_{k=0}^{n} |S_k^5|$  & $/5$ & 0 & 1 & 2 &        2 &  1 &  0 &      0 & 1 & 2 \\
         \hline
         $A_{n-1,n}^5$             & $/5$ & - & 1 & 2 &        1 & -1 & -2 &     -1 & 1 & 2 \\
         \hline
         $A_n^5$                   & $/5$ & 0 & 1 & 3 &        4 &  3 &  1 &      0 & 1 & 3 \\
         \hline
         $\cfrac{A_n^5}{|S_n^5|}$  &      & - & 1 & 3 & $\infty$ & -3 & -1 &      - & 1 & 3 \\
         \hline
    \end{tabular}
    \caption{Sequences in a $5$-uniform simplicial complex}
    \label{tab:5-uniform}
\end{table}
\renewcommand{\arraystretch}{1}

\end{document}